\documentclass[11pt,a4paper]{amsart}
\usepackage[english]{babel}
\usepackage[utf8]{inputenc}  
\usepackage[T1]{fontenc}  
\usepackage{amscd,amssymb,amsmath,latexsym,url,mathrsfs,graphicx}
\usepackage[all,cmtip]{xy}
\usepackage{hyperref} \hypersetup{breaklinks=true}
\usepackage{enumerate}

\newcommand{\Q}{\mathbb{Q}}
\newcommand{\C}{\mathbb C}
\newcommand{\R}{\mathbb{R}}
\newcommand{\Z}{\mathbb{Z}}

\newcommand{\bQ}{\overline{\mathbb Q}}

\newcommand{\F}{\mathcal{F}}

\renewcommand{\P}{{\mathbb{P}}}

\newcommand{\Clf}{\mathscr{C}}

\newcommand{\bfxi}{{\boldsymbol{\xi}}}
\newcommand{\bfalpha}{{\boldsymbol{\alpha}}}
\newcommand{\bfz}{{\boldsymbol{z}}}
\newcommand{\bftheta}{{\boldsymbol{\theta}}}
\newcommand{\bfu}{{\boldsymbol{u}}}
\newcommand{\bfn}{{\boldsymbol{n}}}

\newcommand{\bfp}{{\boldsymbol{p}}}
\newcommand{\bfzero}{{\boldsymbol{0}}}

\newcommand{\bfW}{\boldsymbol{W}}

\newcommand{\wF}{F_0}

\newcommand{\lambdas}{\lambda_{S^1}}
\newcommand{\lambdasn}{\lambda_{(S^1)^N}}

\def\wh#1{{\widehat{#1}}}

\newcommand{\supp}{\operatorname{supp}}

\renewcommand{\L}{\operatorname{L}}

\newcommand{\gal}{\operatorname{Gal}}

\newcommand{\D}{\mathscr{D}}

\renewcommand{\d}{\operatorname{d}}
\newcommand{\dch}{\operatorname{d_{ch}}}
\newcommand{\dsph}{\operatorname{d_{sph}}}

\newcommand{\h}{\operatorname{h}}

\newcommand{\lip}{\operatorname{Lip}}
\newcommand{\lipsph}{\operatorname{Lip_{sph}}}

\newcounter{thm}
\numberwithin{equation}{section}
\numberwithin{thm}{section}

\theoremstyle{definition}

\newtheorem{rem}[thm]{Remark}

\theoremstyle{plain}
\newtheorem{lemma}[thm]{Lemma}
\newtheorem{prop}[thm]{Proposition}
\newtheorem{teor}[thm]{Theorem}

\begin{document}

\title[Quantitative equidistribution of Galois orbits]{Quantitative equidistribution of Galois orbits of small points in the $N$-dimensional torus}

\author[D'Andrea]{Carlos D'Andrea}
\address{Departament de Matem\`atiques i Inform\`atica, Universitat de Barcelona.
Gran Via~585, 08007 Barcelona, Spain}
\email{cdandrea@ub.edu}
\urladdr{\url{http://atlas.mat.ub.es/personals/dandrea/}}

\author[Narváez-Clauss]{Marta Narv\'aez-Clauss}
\address{Departament de Matem\`atiques i Inform\`atica, Universitat de Barcelona.
Gran Via~585, 08007 Barcelona, Spain}
\email{marta.narvaez@ub.edu}

\author[Sombra]{Mart{\'\i}n~Sombra}
\address{Departament de Matem\`atiques i Inform\`atica, Universitat de Barcelona.
Gran Via~585, 08007 Barcelona, Spain \vspace*{-2.5mm}}
\address{ICREA.
Passeig Llu\'is Companys 23, 08010 Barcelona, Spain}
\email{sombra@ub.edu}
\urladdr{\url{http://atlas.mat.ub.es/personals/sombra/}}

\date{\today} 

\thanks{D'Andrea and Narv\'aez-Clauss were partially supported by the MICINN
  research project MTM2010-20279-C02-01 and the MINECO research project MTM2013-40775-P. Sombra was partially supported by the MINECO 
  research projects MTM2012-38122-C03-02 and MTM2015-65361-P}

\begin{abstract}
We present a quantitative version of Bilu's theorem on the limit distribution of Galois orbits of sequences of points of small height in the $N$-dimensional algebraic torus. Our result gives, for a given point, an explicit bound for the discrepancy between its Galois orbit and the uniform distribution on the compact subtorus, in terms of the height and the generalized degree of the point.
\end{abstract}
\maketitle

\section{Introduction} \label{sec:introduction}
One of the first results concerning the distribution of Galois orbits of points of small height in algebraic varieties is due to Bilu \cite{bilu}. It establishes that the Galois orbits of strict sequences of points of small Weil height in an algebraic torus tend to the uniform distribution around the unit polycircle. 

Let us introduce some notation before giving the precise formulation of this result. Fix an algebraic closure $\bQ$ of $\Q$ together with an embedding $\bQ\hookrightarrow\C$. By $\C^\times$ and $\bQ^\times$ we denote the multiplicative groups of $\C$ and $\bQ$, respectively. Let $N\geq1$, the {\em Galois orbit} of a point in $(\bQ^\times)^N$ is its orbit under the action of the {\em absolute Galois group}, $\gal(\bQ/\Q)$.

For a finite set $T\subset(\C^\times)^N$, the discrete probability measure on $(\C^\times)^N$ associated to it is given by
$$\mu_T=\frac1{\#T}\sum_{\alpha\in T}\delta_\alpha,$$
where $\#T$ denotes the cardinality of $T$ and $\delta_\alpha$ the Dirac delta measure on $(\C^\times)^N$ supported on $\alpha$. The {\em unit polycircle} $(S^1)^N$ is the set of points $(z_1,\ldots,z_n)\in\C^N$ such that $|z_1|=\ldots=|z_N|=1$. It is a compact subgroup of $(\C^\times)^N$. We denote by $\lambdasn$ the Haar probability measure of $(S^1)^N$, considered as a measure on $(\C^\times)^N$.

A sequence $(\mu_k)_{k\geq1}$ of probability measures on $(\C^\times)^N$ {\em converges weakly} to a pro-bability measure $\mu$ on $(\C^\times)^N$ if, for every compactly supported continuous function $F:(\C^\times)^N\rightarrow\R$, we have
\begin{equation*}
\lim_{k\to\infty}\int_{(\C^\times)^N}Fd\mu_k=\int_{(\C^\times)^N}Fd\mu.
\end{equation*}

Let $\xi\in\bQ^\times$ and $f_\xi\in\Z[x]$ be the minimal polynomial of $\xi$ over the integers. The {\em Weil height} of $\xi$ is defined as
$$\h(\xi)=\frac{m(f_\xi)}{\deg(\xi)},$$
where $m(f_\xi)$ is the {\em (logarithmic) Mahler measure} of $f_\xi$, given by
$$m(f_\xi)=\frac1{2\pi}\int_0^{2\pi}\log|f_{\xi}(e^{i\theta})|d\theta,$$
and $\deg(\xi)=[\Q(\xi):\Q]$ is the degree of the point $\xi$.

This notion of height extends to $(\bQ^\times)^N$ as follows:
\begin{equation}\label{defh}\h(\bfxi)=\h(\xi_1)+\ldots+\h(\xi_N),\text{ for every }\bfxi=(\xi_1,\ldots,\xi_N)\in(\bQ^\times)^N.\end{equation}

A sequence $(\bfxi_k)_{k\geq1}$ in $(\bQ^\times)^N$ is {\em strict} if, for every proper algebraic subgroup $Y\subset(\bQ^\times)^N$, the cardinality of the set $\{k:\bfxi_k\in Y\}$ is finite.

\begin{teor}\cite[Theorem 1.1]{bilu}\label{teorbilu}
Let $(\bfxi_k)_{k\geq1}$ be a strict sequence in $(\bQ^\times)^N$ such that $\underset{k\to\infty}{\lim}\h(\bfxi_k)=0$. Then we have
$$\lim_{k\to\infty}\mu_{S_k}=\lambdasn,$$
where $\mu_{S_k}$ is the discrete probability measure associated to the Galois orbit $S_k$ of $\bfxi_k$.
\end{teor}

This result was inspired on a previous work of Szpiro, Ullmo and Zhang \cite{SUZ} on the equidistribution of points of small N\'eron-Tate height in Abelian varieties. It was originally motivated by Bogomolov's conjecture, solved in \cite{Ullmo} and \cite{zhang}. The results of Szpiro, Ullmo and Zhang and of Bilu were largely generalized to other heights and places \cite{Rumely,BaHsia,FRL,Ba-Rum,Ch-L06,Yuan,Gubler,BermBouk,Chen,BPRS}. In particular, these results established the equidistribution of Galois orbits of sequences of small points for all places of $\Q$ and heights associated to algebraic dynamical systems. Moreover, this equidistribution phenomenon holds for the bigger set of test functions with logarithmic singularities along divisors with minimal height, see \cite{ct}.

As a general fact, these equidistribution theorems are formulated in a qualitative way, in the sense that no information is provided on the rate of convergence towards the equidistribution. An exception is \cite{FRL}, where a bound for this rate of convergence is given for a large class of heights of points in the projective line and all places of $\Q$. Independently, Petsche \cite{Petsche} gave a quantitative version of Bilu's result for the case of dimension one.

In this paper, we present a quantitative version of Theorem \ref{teorbilu} for the general $N$-dimensional case. In particular, we provide a bound for the integral of a suitable test function with respect to the signed measure defined by the difference of the discrete probability measure associated to the Galois orbit of a point in $(\bQ^\times)^N$ and the measure $\lambdasn$. This bound is given in terms of the height of the point, a higher dimensional generalization of the notion of the degree of an algebraic number, and a constant depending only on the test function. 

To state our main result properly, let us introduce further definitions and notations. For every $\bfn=(n_1,\ldots,n_N)\in\Z^N$, consider the monomial map 
$$\begin{array}{ccl}
\chi^\bfn:(\bQ^\times)^N &\longrightarrow & \bQ^\times \\ 
\bfz=(z_1,\ldots,z_N)&\longmapsto & \chi^\bfn(\bfz)=z_1^{n_1}\ldots z_N^{n_N}.
\end{array}$$
We define the {\em generalized degree} of a point $\bfxi\in(\bQ^\times)^N$ by
\begin{equation}\label{defD}\D(\bfxi)=\min_{\bfn\neq\bfzero}\{\|\bfn\|_1\deg(\chi^\bfn(\bfxi))\},\end{equation}
where $\deg(\chi^\bfn(\bfxi))$ is the degree of the point $\chi^\bfn(\bfxi)\in\bQ^\times$ and $\|\cdot\|_1$ is the $1$-norm on $\C^N$. For a particular choice of $\bfxi$, the generalized degree can be computed with a finite number of operations (Remark \ref{remgendeg}).

Let us identify $(\R/\Z)^N\times\R^N$ and $(\C^\times)^N$ via the logarithmic-polar coordinates change of variables:
\begin{equation*}
\begin{array}{ccc}
 (\R/\Z)^N\times\R^N & \longrightarrow & (\C^\times)^N  \\
 (\bftheta,\bfu)=((\theta_1,\ldots,\theta_N),(u_1,\ldots,u_N)) & \longmapsto & (e^{2\pi i\theta_1+u_1},\ldots, e^{2\pi i\theta_N+u_N}).
\end{array}
\end{equation*}
On $(\R/\Z)^N\times\R^N\simeq(\C^\times)^N$ we consider the translation invariant distance, defined as
\begin{equation*}
\d((\bftheta,\bfu),(\bftheta',\bfu'))=\left(\sum_{l=1}^N\d_{\rm ang}(\theta_l,\theta_l')^2+|u_l-u_l'|^2\right)^{\frac12},
\end{equation*}
where $\d_{\rm ang}(\theta_l,\theta_l')$ is the Euclidean distance in $S^1$ between $e^{2\pi i\theta_l}$ and $e^{2\pi i\theta_l'}$, divided by~$2\pi$.

A function $F:(\C^\times)^N\rightarrow\R$ belongs to the set of test functions $\F$ if it satisfies:
\begin{enumerate}[(i)]
\item $F$ is a Lipschitz function with respect to the distance $\d$;
\item The restriction $\wF=F|_{(S^1)^N}$ is in $\Clf^{N+1}((S^1)^N,\R)$.
\end{enumerate} 

The set $\F$ contains all compactly supported functions in $\Clf^{N+1}((\C^\times)^N,\R)$.

The following is the main result of this paper.

\begin{teor}\label{maintheor}
There is a constant $C\leq64$ such that, for every $\bfxi\in(\bQ^\times)^N$ with $\h(\bfxi)\leq1$ and every $F\in\F$,
\begin{equation*}
\left|\int_{(\C^\times)^N} Fd\mu_S-\int_{(\C^\times)^N} Fd\lambdasn\right|\leq c(F)\left(4\h(\bfxi)+C\frac{\log(\D(\bfxi)+1)}{\D(\bfxi)}\right)^{\frac{1}{2}},
\end{equation*}
where $S$ is the Galois orbit of $\bfxi$, $\mu_S$ the discrete probability measure associated to it and $c(F)$ a positive constant depending only on $F$.
\end{teor}

For every test function $F\in\F$, the function $\wF$, its Fourier transform $\wh{\wF}$, all the first order partial derivatives of $\wF$ and their corresponding Fourier transforms are integrable with respect to a Haar measure (Theorem \ref{teorfun}). In logarithmic-polar coordinates $\wF(\bftheta)=F(\bftheta,\bfzero)$. Then, as shown in the proof of Theorem \ref{maintheor}, the constant $c(F)$ can be bounded by
\begin{equation*}
c(F)\leq2\lip(F)+16\sum_{l=1}^N\left\|\wh{\frac{\partial\wF}{\partial \theta_l}}\right\|_{\L^1},
\end{equation*}
where $\lip(F)$ is the Lipschitz constant of $F$ with respect to the distance $\d$ of $(\C^\times)^N$ and $\|\cdot\|_{\L^1}$ stands for the $\L^1$-norm of a function on the locally compact Abelian group $\Z^N$ with respect to the standard Haar measure.

Our main theorem is a quantitative version of Bilu's result. Indeed, if we consider a strict sequence $(\bfxi_k)_{k\geq1}$ in $(\bQ^\times)^N$ such that $\h(\bfxi_k)\to0$ as $k\to\infty$, we necessarily have that $\D(\bfxi_k)\to\infty$ as $k\to\infty$ (Lemma \ref{strictseqDinfty}). Hence, for every function $F\in\F$, Theorem \ref{maintheor} implies that
$$\lim_{k\to\infty}\int_{(\C^\times)^N}Fd\mu_{S_k}=\int_{(\C^\times)^N}Fd\lambdasn,$$
where $\mu_{S_k}$ is the discrete probability measure associated to the Galois orbit $S_k$ of $\bfxi_k$. Since $\F$ contains a dense subset of the set of compactly supported continuous functions on $(\C^\times)^N$, we deduce Theorem \ref{teorbilu}.

The rate of convergence in Theorem \ref{maintheor} has the expected
exponent $\frac12$ as in Favre and Rivera-Letelier's paper \cite{FRL},
see also Theorem \ref{thmFRLthesisM}.  On the other hand, one could
ask if, for the general $N$-dimensional case, the constant $c(F)$ might be bounded by the Lipschitz constant of
the test function, as in their paper.

The idea of the proof of our result is to reduce the problem, via monomial maps, to the one-dimensional situation as it was done in \cite{bilu,DAnGaSom}. In this setting, we apply Favre and Rivera-Letelier's result (Theorem \ref{thmFRLthesisM}). Then, we lift the obtained quantitative control to the $N$-dimensional torus by applying the Fourier inversion formula and a study of the Fourier-Stieltjes transform of the discrete probability measure associated to the orbit of the point. 

This paper is structured as follows. Section \ref{sec1} contains preliminary theory and general results on Fourier analysis, measures on the Riemann sphere, Galois invariant sets and the generalized degree. In Section \ref{sec2}, we give the proof of Theorem \ref{maintheor}, which is divided in several propositions and lemmas. At the end of the paper there are two appendices, the first one studies the set of test functions $\F$ and the second the Lipschitz constant of an auxiliary function used in Section \ref{sec2}.

\medskip

\noindent {\bf Acknowledgements.} We thank Joaquim Ortega, Juan Rivera-Letelier and the anonymous referee for useful comments and suggestions. The results of this paper are part of the Ph.D. thesis of the second author \cite{thesisM}.

\section{Preliminaries}\label{sec1}

\subsection{Fourier analysis} In this section we review basic concepts of Fourier analysis on $(\R/\Z)^N$, we refer the reader to \cite{rudin} for the proof of the stated results.

Let $p\geq1$. Given a function $H:(\R/\Z)^N\rightarrow\C$, its {\em $\L^p$-norm} is defined by
$$\|H\|_{\L^p}=\left(\int_{(\R/\Z)^N}|H(\bftheta)|^pd\bftheta\right)^{\frac1p}\in\R_{\geq0}\cup\{+\infty\}.$$
We say that $H\in\L^p((\R/\Z)^N)$ if this norm is finite. In particular, the function $H$ is {\em Haar-integrable} if it lies in $\L^1((\R/\Z)^N)$. Similarly, for a function $G:\Z^N\rightarrow\C$, its {\em $\L^p$-norm} is defined by
$$\|G\|_{\L^p}=\left(\sum_{\bfn\in\Z^N}|G(\bfn)|^p\right)^{\frac1p}\in\R_{\geq0}\cup\{+\infty\}$$
and we say that $G\in\L^p(\Z^N)$ if this norm is finite. Also, $G$ is {\em Haar-integrable} if it lies in $\L^1(\Z^N)$.

Let $H:(\R/\Z)^N\rightarrow\C$ be Haar-integrable, its Fourier transform is the function $\wh{H}:\Z^N\rightarrow\C$ defined as
$$\wh{H}(\bfn)=\int_{(\R/\Z)^N}H(\bftheta)e^{-2\pi i\bfn\cdot\bftheta}d\bftheta,$$
where 
$$\bfn\cdot\bftheta=(n_1,\ldots,n_N)\cdot(\theta_1,\ldots,\theta_N)=n_1\theta_1+\cdots+n_N\theta_N.$$
If, in addition, $\wh{H}$ is also Haar-integrable, the {\em Fourier inversion formula} states that 
$$H(\bftheta)=\sum_{\bfn\in\Z^N}\wh{H}(\bfn)e^{2\pi i\bfn\cdot\bftheta}.$$

For $H\in(\L^1\cap\L^2)((\R/\Z)^N)$, {\em Plancherel's theorem} states that $\wh{H}\in\L^2(\Z^N)$ and moreover the following holds
$$\|\wh{H}\|_{\L^2}=\|H\|_{\L^2}.$$

For every finite and regular positive measure $\lambda$ on $(\R/\Z)^N$, its {\em Fourier-Stieltjes transform} is the function $\wh{\lambda}:\Z^N\rightarrow\C$ given by
$$\wh{\lambda}(\bfn)=\int_{(\R/\Z)^N}e^{-2\pi i\bfn\cdot\bftheta}d\lambda(\bftheta).$$

We now establish some auxiliary results that will be useful for the proof of Theorem~\ref{maintheor}.

\begin{lemma}\label{lemmalambda}
Let $H:(\R/\Z)^N\longrightarrow\C$ be a Haar-integrable function such that its Fourier transform $\wh{H}$ is also Haar-integrable. For any finite regular measure $\lambda$ on $(\R/\Z)^N$ we have that $H$ is integrable with respect to $\lambda$ and $\wh{H}\wh{\lambda}$ is Haar-integrable. Moreover, the following holds
$$\int_{(\R/\Z)^N}Hd\lambda=\sum_{\bfn\in\Z^N}\wh{H}(\bfn)\overline{\wh{\lambda}(\bfn)}.$$
\end{lemma}

\begin{proof}
Let $\lambda$ be a finite regular measure on $(\R/\Z)^N$. Its Fourier-Stieltjes transform is the function $\wh{\lambda}:\Z^N\to\C$ given by
$$\wh{\lambda}(\bfn)=\int_{(\R/\Z)^N}e^{-2\pi i\bfn\cdot\bftheta}d\lambda(\bftheta).$$
Since both $H$ and $\wh{H}$ are Haar-integrable, we apply the Fourier inversion formula that, together with Fubini's theorem, leads to
\begin{align*}
\int_{(\R/\Z)^N}Hd\lambda
&=\int_{(\R/\Z)^N}\left(\sum_{\bfn\in\Z^N}\wh{H}(\bfn)e^{2\pi i\bfn\cdot\bftheta}\right)d\lambda(\bftheta)\\
&=\sum_{\bfn\in\Z^N}\wh{H}(\bfn)\left(\int_{(\R/\Z)^N}e^{2\pi i\bfn\cdot\bftheta}d\lambda(\bftheta)\right)\\
&=\sum_{\bfn\in\Z^N}\wh{H}(\bfn)\overline{\wh{\lambda}(\bfn)},
\end{align*}
this equality containing the fact that $H$ is integrable with respect to $\lambda$ and that $\wh{H}\wh{\lambda}$ is Haar-integrable.
\end{proof}

\begin{lemma}\label{difinttwosum}
Let $H:(\R/\Z)^N\rightarrow\C$ be a Haar-integrable function such that $\wh{H}$ is also Haar-integrable, and let $\lambda$ be a finite regular measure on $(\R/\Z)^N$. Then
\begin{equation*}
\int_{(\R/\Z)^N}Hd\lambda-\int_{(\R/\Z)^N}Hd\lambdasn
=\wh{H}(\bfzero)\left(\overline{\wh{\lambda}(\bfzero)}-1\right)+\sum_{\bfn\neq\bfzero}\wh{H}(\bfn)\overline{\wh{\lambda}(\bfn)}.
\end{equation*}
\end{lemma}

\begin{proof}
Since $\lambdasn$ is the Haar probability measure of $(\R/\Z)^N$, for any $\bfn\in\Z^N$ we have
$$\wh{\lambdasn}(\bfn)=\int_{(\R/\Z)^N}e^{-2\pi i\bfn\cdot\bftheta}d\bftheta=\begin{cases}1 & \text{if }\bfn=\bfzero,\\ 0 & \text{otherwise.}\end{cases}$$
Hence, by Lemma \ref{lemmalambda} we obtain
\begin{equation*}
\int_{(\C^\times)^N}Hd\lambdasn=\sum_{\bfn\in\Z^N}\wh{H}(\bfn)\overline{\wh{\lambdasn}(\bfn)}=\wh{H}(\bfzero).
\end{equation*}
Then we have
\begin{multline*}
\int_{(\R/\Z)^N}Hd\lambda-\int_{(\R/\Z)^N}Hd\lambdasn\\
=\left(\sum_{\bfn\in\Z^N}\wh{H}(\bfn)\overline{\wh{\lambda}(\bfn)}\right)-\wh{H}(\bfzero)
=\wh{H}(\bfzero)\left(\overline{\wh{\lambda}(\bfzero)}-1\right)+\sum_{\bfn\neq\bfzero}\wh{H}(\bfn)\overline{\wh{\lambda}(\bfn)}.
\end{multline*}
\end{proof}

\subsection{Galois invariant sets} In this section we work with Galois invariant sets and study their height. For further details on basic Galois theory we refer to \cite{lang} and on heights of points, to \cite{HiDG}.

Let $\xi\in\bQ^\times$ and $f_\xi\in\Z[x]$ be the minimal polynomial of $\xi$ over the integers. Recall that the Weil height of $\xi$ is defined as
$$\h(\xi)=\frac{m(f_\xi)}{\deg(\xi)},$$
where $m(f_\xi)$ is the Mahler measure of $f_\xi$, given by
$$m(f_\xi)=\frac1{2\pi}\int_0^{2\pi}\log|f_{\xi}(e^{i\theta})|d\theta,$$
and $\deg(\xi)=[\Q(\xi):\Q]$ is the degree of the point $\xi$. This notion of height coincides with that in \cite[\S1.5]{HiDG}, which is defined using local decompositions.

Let $T\subset(\bQ^\times)^N$ be a finite Galois-invariant set, its height is defined as
$$\h(T)=\sum_{\bfalpha\in T}\h(\bfalpha),$$
where $\h(\alpha)$ is the height of $\alpha\in(\bQ^\times)^N$ as in \eqref{defh}. In particular, since the height of two Galois conjugate points coincide, if $T\subset(\bQ^\times)^N$ is a Galois orbit of cardinality $D$, we have
$$\h(T)=D\h(\bfalpha),$$
for any $\bfalpha\in T$.

\begin{lemma}\label{lemmaD}
Let $\bfxi=(\xi_1,\ldots,\xi_N)$ in $(\bQ^\times)^N$, $S$ its Galois orbit and set $D=\#S$. Then
\begin{enumerate}
\item \label{it1} $D=[\Q(\xi_1,\ldots,\xi_N):\Q]$,
\item \label{it2} for every $\bfn\in\Z^N$, we have that $\deg(\chi^\bfn(\bfxi))$ divides $D$.
\end{enumerate}
\end{lemma}

\begin{proof}
If a group $G$ acts on a finite set $S$ transitively, then for any $x\in S$ the index of the stabilizer $G_x$ is equal to $\#S$, because the cosets $G/G_x$ stay in a natural one-to-one correspondence with the element of $S$. Applying this to $G=\gal(\bQ/\Q)$, $x=\bfxi$ and $S$ the Galois orbit of $\bfxi$, we find $[\gal(\bQ/\Q):\gal(\bQ/\Q(\bfxi))]=D$, which by Galois theory implies that $[\Q(\bfxi):\Q]=D$, proving the first statement. The second statement is immediate because $\chi^{\bfn}(\bfxi)\in\Q(\bfxi)$.
\end{proof}

\begin{lemma}\label{lemmalogxi}
Let $\xi\in\bQ^\times$, $d=\deg(\xi)$ and $S$ its Galois orbit. Then
$$\frac1d\sum_{\alpha\in S}|\log|\alpha||\leq2\h(\xi).$$
\end{lemma}

\begin{proof}
We have
\begin{multline*}
\frac1d\sum_{\alpha\in S}|\log|\alpha||=\frac1d\sum_{\alpha\in S}\max\{-\log|\alpha|,\log|\alpha|\}\\
=\frac1d\sum_{\alpha\in S}\log\max\left\{\frac1{|\alpha|},|\alpha|\right\}=\frac1d\sum_{\alpha\in S}(\log\max\{1,|\alpha|^2\}-\log|\alpha|).
\end{multline*}
Let $P_\xi(x)=a_dx^d+\ldots+a_0\in\Z[x]$ be the minimal polynomial of $\xi$ over $\Z$. Since $S$ is the Galois orbit of $\xi$, we have
$$P_\xi(x)=a_d\prod_{\alpha\in S}(x-\alpha)\ \text{ and }\ a_0=(-1)^da_d\prod_{\alpha\in S}\alpha.$$

Since $|a_0|$ is a nonzero positive integer, we obtain
\begin{multline*}
\frac1d\sum_{\alpha\in S}(\log\max\{1,|\alpha|^2\}-\log|\alpha|)=\frac1d\sum_{\alpha\in S}\log\max\{1,|\alpha|^2\}+\log\frac{|a_d|}{|a_0|}\\
\leq\frac1d\sum_{\alpha\in S}\log\max\{1,|\alpha|^2\}+\log|a_d|\\
\leq2\left(\frac1d\sum_{\alpha\in S}\log\max\{1,|\alpha|\}+\log|a_d|\right)
=2\h(\xi),
\end{multline*}
where the last equality is given by Jensen's formula for the Mahler measure \cite[Proposition 1.6.5]{HiDG}.
\end{proof}

\begin{lemma}\label{lemmaheightbfxi}
Let $\bfxi_1\in(\bQ^\times)^N$ and consider its Galois orbit $\{\bfxi_1,\ldots,\bfxi_D\}$, where $\bfxi_j=(\xi_{j,1},\ldots,\xi_{j,N})$ for every $j=1,\ldots,D$. Then
$$\frac1D\sum_{l=1}^N\sum_{j=1}^D|\log|\xi_{j,l}||\leq2\h(\bfxi_1).$$
\end{lemma}

\begin{proof} For every $l=1,\ldots,N$, the elements $\xi_{j,l}$ and $\xi_{k,l}$ are conjugates. Let us denote by $S_l$ the Galois orbit of $\xi_{1,l}$. By Lemma \ref{lemmaD}, we have that $\#S_l=\deg(\xi_{1,l})$ divides $D$. This is, there is a positive integer $k_l$ such that $D=\deg(\xi_{1,l})k_l$, where $k_l$ is exactly the number of times each element of the orbit is repeated in $\{\xi_{1,l},\ldots,\xi_{D,l}\}$. We obtain
\begin{multline*}\frac1D\sum_{l=1}^N\sum_{j=1}^D|\log|\xi_{j,l}||=\sum_{l=1}^N\frac{1}{k_l\deg(\xi_{1,l})}\sum_{j=1}^D|\log|\xi_{j,l}||
\\=\sum_{l=1}^N\frac1{\deg(\xi_{1,l})}\sum_{\alpha\in S_l}|\log|\alpha||
\leq\sum_{l=1}^N2\h(\xi_{1,l})=2\h(\bfxi_1),
\end{multline*}
where the inequality follows from Lemma \ref{lemmalogxi}.
\end{proof}

\begin{lemma}\label{countelementorbitbound}
Let $S\subset\bQ^\times$ be a Galois-invariant set of cardinality $D$. For every $0<\delta<1$, we have
$$\#S_\delta<2\left(\log\frac1\delta\right)^{-1}\h(S),$$
where $S_\delta=\{\alpha\in S:|\log|\alpha||>\log\frac1\delta\}$.
\end{lemma}

\begin{proof}
Write $S$ as a finite disjoint union of Galois orbits
$$S=S_1\sqcup\cdots\sqcup S_m.$$
By definition, for any $\alpha\in S_\delta$ we have
$$1<\left(\log\frac1\delta\right)^{-1}|\log|\alpha||.$$
Hence, we obtain
\begin{multline*}
\#S_\delta<\sum_{\alpha\in S_\delta}\left(\log\frac1\delta\right)^{-1}|\log|\alpha||\leq\left(\log\frac1\delta\right)^{-1}\sum_{\alpha\in S}|\log|\alpha||\\
=\left(\log\frac1\delta\right)^{-1}\sum_{l=1}^m\sum_{\alpha\in S_l}|\log|\alpha||\leq\left(\log\frac1\delta\right)^{-1}\sum_{l=1}^m2\h(S_l)=2\left(\log\frac1\delta\right)^{-1}\h(S),
\end{multline*}
where the last inequality holds by Lemma \ref{lemmalogxi}.
\end{proof}

\subsection{The generalized degree} We study now the notion of the generalized degree of a point in the algebraic torus defined in \eqref{defD}. First of all, let us see that in dimension one, it coincides with the notion of the degree of the algebraic number. Let $\xi\in\bQ^\times$, then
$$\D(\xi)=\min_{n\neq0}\{|n|\deg(\xi^n)\}.$$
For every non-zero integer $n$, let $Q_n(x)$ be the minimal polynomial of $\xi^{|n|}$ over $\Z$, of degree $\deg(\xi^{|n|})=\deg(\xi^n)$. By setting $R_n(x)=Q_n(x^{|n|})\in\Z[x]$ we obtain that $R_n(\xi)=0$ and this implies that
$$\deg(\xi)\leq\deg(R_n(x))=|n|\deg(\xi^n).$$
Hence, $\D(\xi)=\deg(\xi)$.

\begin{rem}\label{remgendeg}
For $N\geq1$ and every $\bfxi=(\xi_1,\ldots,\xi_N)$ in $(\bQ^\times)^N$ we have that
$$\D(\bfxi)\leq\min\{\deg(\xi_1),\ldots,\deg(\xi_N)\}.$$
This holds since
$$\{\deg(\xi_1),\ldots,\deg(\xi_N)\}\subset\{\deg(\chi^\bfn(\bfxi)):\bfn\neq\bfzero\}.$$
Thus, for a particular choice of $\bfxi$, the generalized degree can be computed after a finite number of steps by considering all $\bfn\neq\bfzero$ such that $\|\bfn\|_1\leq\min\{\deg(\xi_1),\ldots,\deg(\xi_N)\}$.
\end{rem}

For $N=1$, a strict sequence $(\xi_k)_{k\geq1}$ in $\bQ^\times$ such that $\underset{k\to\infty}{\lim}\h(\xi_k)= 0$ verifies that $\underset{k\to\infty}{\lim}\deg(\xi_k)=\infty$. Indeed, suppose there is some $c>0$ such that $\deg(\xi_k)\leq c$ for every $k\geq0$. By Northcott's theorem \cite[Theorem 1.6.8]{HiDG}, there are only finitely many elements with bounded degree and height. Hence, there is some $\alpha\in\bQ^\times$ such that $\xi_k=\alpha$ for infinitely many $k$'s. Since $\h(\xi_k)$ tends to $0$ as $k$ goes to infinity, by Kronecker's theorem \cite[Theorem 1.5.9]{HiDG} we necessarily have $\h(\alpha)=0$, which implies that $\alpha$ is a root of unity. In particular, there is an infinite subsequence of $(\xi_k)_{k\geq1}$ contained in a proper algebraic subgroup of $\bQ^\times$ which is not possible by the assumption that the sequence is strict.

The following lemma is a generalization to the higher dimensional case of this fact.

\begin{lemma}\label{strictseqDinfty}
Let $(\bfxi_k)_{k\geq1}$ be a strict sequence in $(\bQ^\times)^N$ such that $\lim_{k\to\infty}\h(\bfxi_k)=0$. Then
$$\lim_{k\to\infty}\D(\bfxi_k)=\infty.$$
\end{lemma}

\begin{proof}
Since the sequence $(\bfxi_k)_{k\geq0}$ is strict, the sequence $(\chi^\bfn(\bfxi_k))_{k\geq0}$ is a strict sequence in $\bQ^\times$ for every $\bfn\neq\bfzero$.

Write $\bfxi_k=(\xi_{k,1},\ldots,\xi_{k,N})$ and let $\bfn=(n_1,\ldots,n_N)\neq\bfzero$. We have
\begin{multline*}
\h(\chi^\bfn(\bfxi_k))=\h(\xi_{k,1}^{n_1}\cdots\xi_{k,N}^{n_N})\leq\h(\xi_{k,1}^{n_1})+\ldots+\h(\xi_{k,N}^{n_N})\\
=|n_1|\h(\xi_{k,1})+\ldots+|n_N|\h(\xi_{k,N})\leq\|\bfn\|_1\h(\bfxi_k)\xrightarrow[k\to\infty]{}0,
\end{multline*}
where the inequality follows from \cite[\S 1.5.14]{HiDG}.

Thus, as we just saw, we have that for every $\bfn\neq\bfzero$
$$\lim_{k\to\infty}\deg(\chi^\bfn(\bfxi_k))=\infty.$$
Finally, by Remark \ref{remgendeg}, for every $k\geq0$ there is $\bfn_k\neq\bfzero$ with bounded $1$-norm such that $\D(\bfxi_k)=\|\bfn_k\|_1\deg(\chi^{\bfn_k}(\bfxi_k))$. Hence
$$\lim_{k\to\infty}\D(\bfxi_k)=\infty,$$
completing the proof.
\end{proof}

\section{Proof of the main result}\label{sec2}

In this section we give the proof of Theorem \ref{maintheor}. As we mentioned in the introduction, we do so by using Fourier analysis techniques and reducing the problem, via projections, to the one-dimensional case, where the result follows from Favre and Rivera-Letelier's \cite[Corollary 1.4]{FRL}. 

Before stating this result, we give the definition of the spherical distance on the Riemann sphere. Let us identify the projective complex line with the unit sphere $S^2$ of $\R^3$. Let $S^2\setminus\{(0,0,1)\}\rightarrow\C$ be the stereographic projection, where we identify the equator of $S^2$ with the set $\{z\in\C:|z|=1\}$. Composing it with the standard inclusion $\C\hookrightarrow\P^1(\C)$ gives a map $S^2\setminus\{(0,0,1)\}\to\P^1(\C)\setminus\{(0:1)\}$, that we extend to a homeomorphism $\rho:S^2\to\P^1(\C)$ by setting $\rho(0,0,1)=(0:1)$. The spherical distance $\dsph$ on $\P^1(\C)$ is given by the length of the arc on $S^2$ under this identification and extended to $\P^1(\C)^N$ for $\bfp=(p_1,\ldots,p_N)$ and $\bfp'=(p_1',\ldots,p_N')$ as
$$\dsph(\bfp,\bfp')=\left(\sum_{l=1}^N\dsph(p_l,p_l')^2\right)^{\frac12}.$$

A function $f:\P^1(\C)^N\rightarrow\C$ is a {\em Lipschitz function} with respect to the distance $\dsph$ if there is a constant $K\geq0$ such that
\begin{equation}\label{lipfun}|f(\bfp)-f(\bfp')|\leq K\dsph(\bfp,\bfp')\text{ for every }\bfp,\bfp'\in\P^1(\C)^N.\end{equation}
If $f$ is a Lipschitz function with respect to the spherical distance, its {\em Lipschitz constant} $\lipsph(f)$ is the smallest $K\geq0$ such that \eqref{lipfun} holds.

We now state Favre and Rivera-Letelier's result together with the explicit constants computed in the second author's Ph.D. Thesis \cite[Theorem II]{thesisM}.

\begin{teor}\label{thmFRLthesisM}
There is a positive constant $C_0\leq15$ such that for every $\Clf^1$- function $f:\P^1(\C)\rightarrow\R$ and every $\xi\in\bQ^\times$
\begin{equation*}
\left|\int_{\P^1(\C)}fd\mu_S-\int_{\P^1(\C)}fd\lambdas\right|\leq\lipsph(f)\left(\frac{\pi}{\deg(\xi)}+\left(4\h(\xi)+C_0\frac{\log(\deg(\xi)+1)}{\deg(\xi)}\right)^{\frac12}\right),
\end{equation*}
where $S$ is the Galois orbit of $\xi$, $\mu_S$ the discrete probability measure associated to it and $\lipsph$ stands for the Lipschitz constant with respect to the spherical distance on the Riemann sphere.

In particular, if $\h(\xi)\leq1$, then
\begin{equation*}
\left|\int_{\P^1(\C)}fd\mu_S-\int_{\P^1(\C)}fd\lambdas\right|\leq\lipsph(f)\left(4\h(\xi)+C\frac{\log(\deg(\xi)+1)}{\deg(\xi)}\right)^{\frac12},
\end{equation*}
for $C\leq64$.
\end{teor}

The proof of this result relies on the interpretation of the height of a point in terms of the potential theory over the complex projective line. Given $\xi\in\bQ^\times$, it can be shown that the mutual energy of the signed measure $\mu_S-\lambdas$ is bounded above by twice the height of the point. Since this signed measure is not regular enough, Favre and Rivera-Letelier consider a regularization such that it has vanishing total mass and its trace measure has continuous potential. This allows to apply a Cauchy-Schwartz type inequality to the integral of the function with respect to the regularized measure. Together with the study of the integral of the function with respect to the difference of the measure and its regularization, this leads to their result. The explicitation of the constant in \cite{thesisM} is done by considering a specific regularization of the measure, which is done by convolution with an specific {\em mollifier}.

Consider the projection
\begin{equation*}
\begin{array}{cccc}
 \pi: & (\R/\Z)^N\times\R^N & \longrightarrow & (\R/\Z)^N \\
  & (\bftheta,\bfu) & \longmapsto & \bftheta.   
\end{array}
\end{equation*}
Under the natural identifications
\begin{equation*}
\begin{array}{ccc}
 (\C^\times)^N & \longrightarrow & (\R/\Z)^N\times\R^N \\
 (z_1,\ldots,z_N) & \longmapsto & \left(\left(\frac{\arg(z_1)}{2\pi},\ldots,\frac{\arg(z_N)}{2\pi}\right),(\log|z_1|,\ldots,\log|z_N|)\right)
\end{array}
\end{equation*}
\begin{center} and \end{center}
\begin{equation*}
\begin{array}{ccc}
 (S^1)^N & \longrightarrow & (\R/\Z)^N \\
 (z_1,\ldots,z_N) & \longmapsto & \left(\frac{\arg(z_1)}{2\pi},\ldots,\frac{\arg(z_N)}{2\pi}\right),
\end{array}
\end{equation*}
the map $\pi$ can be re-written as 
\begin{equation*}
\begin{array}{cccc}
 \pi: & (\C^\times)^N & \longrightarrow & (S^1)^N  \\
  & (z_1,\ldots,z_N) & \longmapsto & \left(\frac{z_1}{|z_1|},\ldots,\frac{z_N}{|z_N|}\right).
\end{array}
\end{equation*}

Let $\bfxi\in(\Q^\times)^N$, $S$ its Galois orbit and $\mu_S$ the discrete probability measure associated to it. If $F:(\R/\Z)^N\times\R^N\longrightarrow\C$ is integrable with respect to the measure $\lambdasn$, then we have
\begin{multline}\label{divideprooftwosum}
\left|\int_{(\R/\Z)^N\times\R^N}Fd\mu_S-\int_{(\R/\Z)^N\times\R^N}Fd\lambdasn\right|\\
\leq\left|\int_{(\R/\Z)^N\times\R^N}Fd\mu_S-\int_{(\R/\Z)^N}\wF d\nu_S\right|\\
+\left|\int_{(\R/\Z)^N}\wF d\nu_S-\int_{(\R/\Z)^N\times\R^N}Fd\lambdasn\right|,
\end{multline}
where $\wF:(\R/\Z)^N\to\R$ is defined by $\wF(\bftheta)=F(\bftheta,\bfzero)$, and the measure $\nu_S$ is the pushforward of the measure $\mu_S$, which is given by
\begin{equation}\label{eqnus}\nu_S=\pi_*\mu_S=\frac1{\#S}\sum_{\bfalpha\in S}\delta_{\frac{\alpha}{|\alpha|}}.\end{equation}

Using \eqref{divideprooftwosum}, we are able to divide the proof of the main result into two parts. The following proposition corresponds to the first one.

\begin{prop}\label{propsum1}
Let $\bfxi\in(\bQ^\times)^N$ and $S$ its Galois orbit. Let $F:(\C^\times)^N\longrightarrow\R$ be a Lipschitz function with respect to the distance $\d$ and such that it is integrable with respect to $\lambdasn$, then
\begin{equation*}
\left|\int_{(\R/\Z)^N\times\R^N}Fd\mu_S-\int_{(\R/\Z)^N}\wF d\nu_S\right|\leq2\lip(F)\h(\bfxi),
\end{equation*}
where $\wF(\bftheta)=F(\bftheta,\bfzero)$ and $\lip(F)$ is the Lipschitz constant of $F$.
\end{prop}

\begin{proof}
With the above notation, we have
\begin{align*}
\left|\int_{(\R/\Z)^N\times\R^N}Fd\mu_S\right.&\left.-\int_{(\R/\Z)^N}\wF d\nu_S\right|=\left|\int_{(\R/\Z)^N\times\R^N}Fd\mu_S-\int_{(\R/\Z)^N\times\R^N}(F\circ\pi) d\mu_S\right|\\
&=\left|\int_{(\C^\times)^N}\left(F(z_1,\ldots,z_N)-F\left(\frac{z_1}{|z_1|},\ldots,\frac{z_N}{|z_N|}\right)\right)d\mu_S(z_1,\ldots,z_N)\right|\\
&\leq\frac1{\#S}\sum_{(\alpha_1,\ldots,\alpha_N)\in S}\left|F(\alpha_1,\ldots,\alpha_N)-F\left(\frac{\alpha_1}{|\alpha_1|},\ldots,\frac{\alpha_N}{|\alpha_N|}\right)\right|\\
&\leq\frac1{\#S}\lip(F)\sum_{(\alpha_1,\ldots,\alpha_N)\in S}\d\left((\alpha_1,\ldots,\alpha_N),\left(\frac{\alpha_1}{|\alpha_1|},\ldots,\frac{\alpha_N}{|\alpha_N|}\right)\right),
\end{align*}
where the last inequality is given by the fact that $F$ is a Lipschitz function with respect to the distance $\d$ of $(\C^\times)^N$. By the definition of this distance, we have
\begin{equation*}
\d\left((\alpha_1,\ldots,\alpha_N),\left(\frac{\alpha_1}{|\alpha_1|},\ldots,\frac{\alpha_N}{|\alpha_N|}\right)\right)=\left(\sum_{l=1}^N|\log|\alpha_l||^2\right)^{\frac12}\leq\sum_{l=1}^N|\log|\alpha_l||.
\end{equation*}
Hence, by Lemma \ref{lemmaheightbfxi}, we conclude
\begin{multline*}
\left|\int_{(\R/\Z)^N\times\R^N}Fd\mu_S-\int_{(\R/\Z)^N}\wF d\nu_S\right|\\
\leq\frac1{\#S}\lip(F)\sum_{(\alpha_1,\ldots,\alpha_N)\in S}\sum_{l=1}^N|\log|\alpha_l||\leq2\lip(F)\h(\bfxi).
\end{multline*}
\end{proof}

Let us study now the second summand in \eqref{divideprooftwosum}. First of all we observe that, since the measure $\lambdasn$ is supported on $(\R/\Z)^N\times\{\bfzero\}$, we can reduce the problem to the compact torus $(S^1)^N$. Indeed, with the notation as in \eqref{divideprooftwosum}, we have
\begin{equation*}
\left|\int_{(\R/\Z)^N}\wF d\nu_S-\int_{(\R/\Z)^N\times\R^N}Fd\lambdasn\right|=\left|\int_{(\R/\Z)^N}\wF d\nu_S-\int_{(\R/\Z)^N}\wF d\lambdasn\right|,
\end{equation*}
where $\nu_S$ is given by \eqref{eqnus}.

If $\wF:(\R/\Z)^N\longrightarrow\R$ is Haar-integrable and such that its Fourier transform $\wh{\wF}$ is also Haar-integrable, by Lemma \ref{difinttwosum} we have
\begin{equation*}
\int_{(\R/\Z)^N}\wF d\nu_S-\int_{(\R/\Z)^N}\wF d\lambdasn=\wh{\wF}(\bfzero)\left(\overline{\wh{\nu}_S(\bfzero)}-1\right)+\sum_{\bfn\neq\bfzero}\wh{\wF}(\bfn)\overline{\wh{\nu}_S(\bfn)},
\end{equation*}
where the Fourier-Stieltjes transform of $\nu_S$ is
\begin{equation}\label{FSnuS}
\wh{\nu}_S(\bfn)=\int_{(\R/\Z)^N}e^{-2\pi i\bfn\cdot\bftheta}d\nu_S(\bftheta)=\frac1{\#S}\sum_{(\alpha_1,\ldots,\alpha_N)\in S}e^{-i\bfn\cdot(\arg(\alpha_1),\ldots,\arg(\alpha_N))},
\end{equation}
for every $\bfn\in\Z^N$. In particular, $\wh{\nu}_S(\bfzero)=1$.

We obtain the following lemma.

\begin{lemma}\label{lemmawF}
Let $\wF:(\R/\Z)^N\longrightarrow\R$ be Haar-integrable and such that its Fourier transform is also Haar-integrable. With the notation as above, we have
\begin{equation*}
\int_{(\R/\Z)^N}\wF d\nu_S-\int_{(\R/\Z)^N}\wF d\lambdasn=\sum_{\bfn\neq\bfzero}\wh{\wF}(\bfn)\overline{\wh{\nu}_S(\bfn)}.
\end{equation*}
\end{lemma}

We study now the Fourier-Stieltjes transform of the measure $\nu_S=\pi_*\mu_S$.

\begin{prop}\label{propnus}
There is a constant $C\leq64$ such that, for every $\bfn\neq\bfzero$ and every $0<\delta<1$, if $\h(\bfxi)\leq1,$ we have
\begin{equation*}
\left|\wh{\nu}_S(\bfn)\right|\leq\frac{-2}{\log\delta}\|\bfn\|_1\h(\bfxi)+\frac{4\sqrt{2}(\delta^2+9)}{\delta^3}\|\bfn\|_1\left(4\h(\bfxi)+C\frac{\log(\D(\bfxi)+1)}{\D(\bfxi)}\right)^{\frac12}.
\end{equation*}
\end{prop}

\begin{proof}
Let $\bfn\neq\bfzero$ and let $S_\bfn$ be the Galois orbit of $\chi^\bfn(\bfxi)$. By Lemma \ref{lemmaD}, there is an integer $l_\bfn$ such that $\#S=l_\bfn\#S_\bfn$ and we know that every element $\alpha\in S_\bfn$ is repeated $l_\bfn$ times in $\{\chi^\bfn(\bfalpha):\bfalpha\in S\}$. Hence, by \eqref{FSnuS}, we obtain
\begin{equation*}
\overline{\wh{\nu}_S(\bfn)}=\frac1{\#S}\sum_{(\alpha_1,\ldots,\alpha_N)\in S}e^{i\bfn\cdot(\arg(\alpha_1),\ldots,\arg(\alpha_N))}=\frac1{\#S}\sum_{\bfalpha\in S}\frac{\chi^\bfn(\bfalpha)}{|\chi^\bfn(\bfalpha)|}=\frac1{\#S_\bfn}\sum_{\alpha\in S_\bfn}\frac{\alpha}{|\alpha|}.
\end{equation*}
For $0<\delta<1$, consider the function $f_\delta:\P^1(\C)\rightarrow\C$ defined as follows.
$$f_\delta(0:1)=0,\ f_\delta(1:z)=\rho_\delta(|z|)\frac{z}{|z|}\text{ for any }z\in\C,$$
where the function $\rho_\delta:\R\rightarrow[0,1]$ is given by
$$\rho_\delta(r)=\begin{cases}0 & \text{if } r<\frac\delta2,\\ \frac{(5\delta-4r)(\delta-2r)^2}{\delta^3} & \text{if } \frac\delta2\leq r\leq\delta,\\ 1 & \text{if } \delta<r<\frac1\delta,\\ (-2+\delta r)^2(-1+2\delta r) & \text{if } \frac1\delta\leq r\leq\frac2\delta,\\ 0 & \text{if } r>\frac2\delta.\end{cases}$$
In Lemma \ref{lemmafdelta}, we prove that $f_\delta$ is a $\Clf^1$-function such that, if we write $f_\delta=u_\delta+iv_\delta$ we have
$$\lipsph(u_\delta),\lipsph(v_\delta)\leq\frac{2\sqrt{2}(\delta^2+9)}{\delta^3},$$
where $\lipsph$ stands for the Lipschitz constant with respect to the spherical distance on the Riemann sphere.

For every $\bfn\neq\bfzero$, we have
\begin{align}\label{eqmu1}
\left|\overline{\wh{\nu}_S(\bfn)}- \right.&\left.\frac1{\#S}\sum_{\bfalpha\in S}^Df_\delta(1:\chi^\bfn(\bfalpha))\right| \\
\nonumber&=\left|\frac1{\#S}\sum_{\bfalpha\in S}\frac{\chi^\bfn(\bfalpha)}{|\chi^\bfn(\bfalpha)|}-\frac1{\#S}\sum_{\bfalpha\in S}\rho_\delta(|\chi^\bfn(\bfalpha)|)\frac{\chi^\bfn(\bfalpha)}{|\chi^\bfn(\bfalpha)|}\right|\\
\nonumber&=\left|\frac1{\#S}\sum_{\bfalpha\in S}\frac{\chi^\bfn(\bfalpha)}{|\chi^\bfn(\bfalpha)|}\left(1-\rho_\delta(|\chi^\bfn(\bfalpha)|)\right)\right|\\
\nonumber&\leq\frac1{\#S}\sum_{\bfalpha\in S}\left|1-\rho_\delta(|\chi^\bfn(\bfalpha)|)\right|.
\end{align}

Let us define, for every $\bfn\neq\bfzero$ and $0<\delta<1$, the set
$$J_{\bfn,\delta}=\left\{\bfalpha\in S:\delta\leq|\chi^\bfn(\bfalpha)|\leq\frac1\delta\right\}.$$
If $\bfalpha\in J_{\bfn,\delta}$, then $\rho_\delta(|\chi^\bfn(\bfalpha)|)=1$, and $0\leq\rho_\delta(|\chi^\bfn(\bfalpha)|)<1$ otherwise. Hence, we have
\begin{equation}\label{eqmu2}
\frac1{\#S}\sum_{\bfalpha\in S}\left|1-\rho_\delta(|\chi^\bfn(\bfalpha)|)\right|=\frac1{\#S}\sum_{\bfalpha\notin J_{\bfn,\delta}}1-\rho_\delta(|\chi^\bfn(\bfalpha)|)\leq\frac1{\#S}\sum_{\bfalpha\notin J_{\bfn,\delta}}1.
\end{equation}
Set 
$$S_{\bfn,\delta}=\left\{\alpha\in S_\bfn:|\log|\bfalpha||>\log\frac1\delta\right\},$$
then we obtain
\begin{equation}\label{eqmu3}
\frac1{\#S}\sum_{\bfalpha\notin J_{\bfn,\delta}}1=\frac1{\# S_\bfn}\sum_{\alpha\in S_{\bfn,\delta}}1\leq2\left(\log\frac1\delta\right)^{-1}\h(\chi^\bfn(\bfxi)),
\end{equation}
where the last inequality is given by Lemma \ref{countelementorbitbound}.

As we saw in the proof of Lemma \ref{strictseqDinfty}, for $\bfn\neq\bfzero$ we have
$$\h(\chi^\bfn(\bfxi))\leq\|\bfn\|_1\h(\bfxi).$$
Thus, putting this together with \eqref{eqmu1}, \eqref{eqmu2} and \eqref{eqmu3} we deduce that
\begin{equation}\label{muminusfdelta}
\left|\overline{\wh{\nu}_S(\bfn)}-\frac1{\#S}\sum_{\bfalpha\in S}f_\delta(1:\chi^\bfn(\bfalpha))\right|\leq2\left(\log\frac1\delta\right)^{-1}\|\bfn\|_1\h(\bfxi).
\end{equation}

On the other hand, we have
\begin{equation}\label{fdelta1}
\frac1{\#S}\sum_{\bfalpha\in S}f_\delta(1:\chi^\bfn(\bfalpha))=\frac1{l_\bfn\#S_\bfn}\sum_{\alpha\in S_\bfn}l_\bfn f_\delta(1:\alpha)=\int_{\P^1(\C)}f_\delta d\mu_{S_\bfn},
\end{equation}
where $\mu_{S_\bfn}$ is the discrete probability measure on $\P^1(\C)$ associated to the Galois orbit $S_\bfn$ of $\chi^\bfn(\bfxi)$.

Since $\lambdas$ is the measure on $\P^1(\C)$ supported on the unit circle, where it coincides with the Haar probability measure and, by definition, $f_\delta(1:z)=z$ if $|z|=1$, we have
$$\int_{\P^1(\C)}f_\delta d\lambdas=\int_{\C^\times}zd\lambdas(z)=0.$$

By Theorem \ref{thmFRLthesisM}, we obtain
\begin{multline}\label{fdelta2}
\left|\int_{\P^1(\C)}f_\delta d\mu_{S_\bfn}\right|=\left|\int_{\P^1(\C)}f_\delta d\mu_{S_\bfn}-\int_{\P^1(\C)}f_\delta d\lambdas \right|\\
\leq\left|\int_{\P^1(\C)}u_\delta d\mu_{S_\bfn}-\int_{\P^1(\C)}u_\delta d\lambdas\right|+\left|\int_{\P^1(\C)}v_\delta d\mu_{S_\bfn}-\int_{\P^1(\C)}v_\delta d\lambdas\right|\\
\leq(\lipsph(u_\delta)+\lipsph(v_\delta))\left(\frac{\pi}{\deg(\chi^\bfn(\bfxi))}+\left(4\h(\chi^\bfn(\bfxi))+C_0\frac{\log(\deg(\chi^\bfn(\bfxi))+1)}{\deg(\chi^\bfn(\bfxi))}\right)^{\frac12}\right)\\
\leq\frac{4\sqrt{2}(\delta^2+9)}{\delta^3}\left(\frac{\pi}{\deg(\chi^\bfn(\bfxi))}+\left(4\h(\chi^\bfn(\bfxi))+C_0\frac{\log(\deg(\chi^\bfn(\bfxi))+1)}{\deg(\chi^\bfn(\bfxi))}\right)^{\frac12}\right),
\end{multline}
where $C_0\leq15$.

Since $\h(\chi^\bfn(\bfxi))\leq\|\bfn\|_1\h(\bfxi)$, we have
\begin{multline*}
\left(4\h(\chi^\bfn(\bfxi))+C_0\frac{\log(\deg(\chi^\bfn(\bfxi))+1)}{\deg(\chi^\bfn(\bfxi))}\right)^{\frac12}\\
\leq\left(4\|\bfn\|_1\h(\bfxi)+C_0\frac{\|\bfn\|_1\log(\deg(\chi^\bfn(\bfxi))+1)}{\|\bfn\|_1\deg(\chi^\bfn(\bfxi))}\right)^{\frac12}\\
\leq\sqrt{\|\bfn\|_1}\left(4\h(\bfxi)+C_0\frac{\log(\|\bfn\|_1\deg(\chi^\bfn(\bfxi))+1)}{\|\bfn\|_1\deg(\chi^\bfn(\bfxi))}\right)^{\frac12}\\
\leq\|\bfn\|_1\left(4\h(\bfxi)+C_0\frac{\log(\|\bfn\|_1\deg(\chi^\bfn(\bfxi))+1)}{\|\bfn\|_1\deg(\chi^\bfn(\bfxi))}\right)^{\frac12}.
\end{multline*}
Hence, this together with \eqref{fdelta1} and \eqref{fdelta2} gives
\begin{align}\label{fdelta3}
\left|\int_{\P^1(\C)}f_\delta d\mu_{S_\bfn}\right|
\leq &\frac{4\sqrt{2}(\delta^2+9)}{\delta^3}\|\bfn\|_1\left(\frac{\pi}{\|\bfn\|_1\deg(\chi^\bfn(\bfxi))}\right.\\
\nonumber & \left.+\left(4\h(\bfxi)+C_0\frac{\log(\|\bfn\|_1\deg(\chi^\bfn(\bfxi))+1)}{\|\bfn\|_1\deg(\chi^\bfn(\bfxi))}\right)^{\frac12}\right)\\
\nonumber\leq & \frac{4\sqrt{2}(\delta^2+9)}{\delta^3}\|\bfn\|_1\left(4\h(\bfxi)+C\frac{\log(\|\bfn\|_1\deg(\chi^\bfn(\bfxi))+1)}{\|\bfn\|_1\deg(\chi^\bfn(\bfxi))}\right)^{\frac12},
\end{align}
with $C\leq64$.

The function $\frac{\log(x+1)}{x}$ is monotonically decreasing for $x\geq1$. We deduce that, for every $\bfn\neq\bfzero$
$$\frac{\log(\|\bfn\|_1\deg(\chi^\bfn(\bfxi))+1)}{\|\bfn\|_1\deg(\chi^\bfn(\bfxi))}\leq\frac{\log(\D(\bfxi)+1)}{\D(\bfxi)}.$$
Together with \eqref{fdelta3}, this implies that, for every $\bfn\neq\bfzero$,
$$\left|\int_{\P^1(\C)}f_\delta d\mu_{S_\bfn}\right|\leq\frac{4\sqrt{2}(\delta^2+9)}{\delta^3}\|\bfn\|_1\left(4\h(\bfxi)+C\frac{\log(\D(\bfxi)+1)}{\D(\bfxi)}\right)^{\frac12}.$$

By using this inequality and \eqref{eqmu1} we deduce that 
\begin{align*}
\left|\overline{\wh{\nu}_S(\bfn)}\right|&\leq\left|\overline{\wh{\nu}_S(\bfn)}-\frac1{\#S}\sum_{\bfalpha\in S}f_\delta(1:\chi^\bfn(\bfalpha))\right|+\left|\frac1{\#S}\sum_{\bfalpha\in S}f_\delta(1:\chi^\bfn(\bfalpha))\right|\\
&\leq \frac{-2}{\log\delta}\|\bfn\|_1\h(\bfxi)+\frac{4\sqrt{2}(\delta^2+9)}{\delta^3}\|\bfn\|_1\left(4\h(\bfxi)+C\frac{\log(\D(\bfxi)+1)}{\D(\bfxi)}\right)^{\frac12},
\end{align*}
proving the proposition.
\end{proof}

The following proposition bounds the second summand in the inequality \eqref{divideprooftwosum}.

\begin{prop}\label{propsum2}
There is a constant $C\leq64$ such that, for every $\bfxi\in(\bQ^\times)^N$ with $\h(\bfxi)\leq1$, every $0<\delta<1$ and every $F\in\F$, the following holds
\begin{multline*}
\left|\int_{(\R/\Z)^N}\wF d\nu_S-\int_{(\R/\Z)^N\times\R^N}Fd\lambdasn\right|\\
\leq\frac1{2\pi}\left(\frac{-2}{\log\delta}+\frac{4\sqrt{2}(\delta^2+9)}{\delta^3}\right)\left(4\h(\bfxi)+C\frac{\log(\D(\bfxi)+1)}{\D(\bfxi)}\right)^{\frac12}\sum_{l=1}^N\left\|\wh{\frac{\partial \wF}{\partial \theta_l}}\right\|_{\L^1},
\end{multline*}
where $S$ is the Galois orbit of $\bfxi$, $\mu_S$ the discrete probability measure associated to it, $\D(\bfxi)$ the generalized degree of $\bfxi$ and $\wF(\bftheta)=F(\bftheta,\bfzero)$.
\end{prop}

\begin{proof}
In Appendix \ref{appendixtestfun} we prove that, given $F\in\F$, the function $\wF$ is Haar-integrable as well as its Fourier transform $\wh{\wF}$. Thus, by Lemma \ref{lemmawF} and Proposition \ref{propnus},
\begin{multline*}
\left|\int_{(\R/\Z)^N}\wF d\nu_S-\int_{(\R/\Z)^N\times\R^N}Fd\lambdasn\right|=\left|\int_{(\R/\Z)^N}\wF d\nu_S-\int_{(\R/\Z)^N}\wF d\lambdasn\right|\\
=\left|\sum_{\bfn\neq\bfzero}\wh{\wF}(\bfn)\overline{\wh{\nu}_S(\bfn)}\right|\leq\sum_{\bfn\neq\bfzero}\left|\wh{\wF}(\bfn)\right|\left|\wh{\nu}_S(\bfn)\right|\\\leq\sum_{\bfn\neq\bfzero}\left|\wh{\wF}(\bfn)\right|\left(\frac{-2}{\log\delta}\|\bfn\|_1\h(\bfxi)+\frac{4\sqrt{2}(\delta^2+9)}{\delta^3}\|\bfn\|_1\left(4\h(\bfxi)+C\frac{\log(\D(\bfxi)+1)}{\D(\bfxi)}\right)^{\frac12}\right)\\
\leq\left(\frac{-2}{\log\delta}+\frac{4\sqrt{2}(\delta^2+9)}{\delta^3}\right)\left(4\h(\bfxi)+C\frac{\log(\D(\bfxi)+1)}{\D(\bfxi)}\right)^{\frac12}\sum_{\bfn\neq\bfzero}\left|\wh{\wF}(\bfn)\right|\|\bfn\|_1,
\end{multline*}
where the last inequality is given by the fact that $\h(\bfxi)\leq1$.

By Lemma \ref{lemmatrans}, for every $l=1,\ldots,N$ we have
$$\wh{\frac{\partial \wF}{\partial \theta_l}}(\bfn)=2\pi in_l\wh{\wF}(\bfn).$$
Hence, we obtain
\begin{multline*}
\sum_{\bfn\neq\bfzero}\left|\wh{\wF}(\bfn)\right|\|\bfn\|_1=\frac1{2\pi}\sum_{l=1}^N\sum_{\bfn\neq\bfzero}\left|\wh{\wF}(\bfn)\right|\cdot|2\pi n_l|\\
=\frac1{2\pi}\sum_{l=1}^N\sum_{\bfn\in\Z^N}\left|\wh{\frac{\partial \wF}{\partial\theta_l}}(\bfn)\right|=\frac1{2\pi}\sum_{l=1}^N\left\|\wh{\frac{\partial \wF}{\partial\theta_l}}\right\|_{\L^1}.
\end{multline*}

Finally, we conclude:
\begin{multline*}
\left|\int_{(\R/\Z)^N}\wF d\pi_*\mu_S-\int_{(\R/\Z)^N\times\R^N}Fd\lambdasn\right|\\
\leq\left(\frac{-2}{\log\delta}+\frac{4\sqrt{2}(\delta^2+9)}{\delta^3}\right)\left(4\h(\bfxi)+C\frac{\log(\D(\bfxi)+1)}{\D(\bfxi)}\right)^{\frac12}\frac1{2\pi}\sum_{l=1}^N\left\|\wh{\frac{\partial \wF}{\partial\theta_l}}\right\|_{\L^1}
\end{multline*}
\end{proof}

\begin{proof}[Proof of Theorem \ref{maintheor}]
Let $F\in\F$ and set $\wF(\bftheta)=F(\bftheta,\bfzero)$. By Theorem \ref{teorfun}, the function $\wF$ and its Fourier transform $\wh{\wF}$ are Haar-integrable and thus, as shown in \eqref{divideprooftwosum}, we have
\begin{multline*}
\left|\int_{(\R/\Z)^N\times\R^N}Fd\mu_S-\int_{(\R/\Z)^N\times\R^N}Fd\lambdasn\right|\\
\leq\left|\int_{(\R/\Z)^N\times\R^N}Fd\mu_S-\int_{(\R/\Z)^N}\wF d\nu_S\right|\\
+\left|\int_{(\R/\Z)^N}\wF d\nu_S-\int_{(\R/\Z)^N\times\R^N}Fd\lambdasn\right|.
\end{multline*}

Since the function $F$ is Lipschitz with respect to the distance $\d$, by Propositions \ref{propsum1} and \ref{propsum2}, there is a constant $C\leq64$ such that
\begin{multline*}
\left|\int_{(\R/\Z)^N\times\R^N}Fd\mu_S-\int_{(\R/\Z)^N\times\R^N}Fd\lambdasn\right|\leq2\lip(F)\h(\bfxi)\\
+\frac1{2\pi}\left(\frac{-2}{\log\delta}+\frac{4\sqrt{2}(\delta^2+9)}{\delta^3}\right)\left(4\h(\bfxi)+C\frac{\log(\D(\bfxi)+1)}{\D(\bfxi)}\right)^{\frac12}\sum_{l=1}^N\left\|\wh{\frac{\partial \wF}{\partial\theta_l}}\right\|_{\L^1}.
\end{multline*}
By Theorem \ref{teorfun}, the Fourier transforms of the first order partial derivatives of $\wF$ are Haar-integrable and so this bound is finite.

We search numerically for the minimum of the function $\frac{-2}{\log\delta}+\frac{4\sqrt{2}(\delta^2+9)}{\delta^3}$, for $0<\delta<1$, and we obtain the value $94.9591$, attained at $\delta\approx0.9071$. Hence, since $\h(\bfxi)\leq1$ and $94.9591/2\pi<16$, we have
\begin{multline*}
\left|\int_{(\R/\Z)^N\times\R^N}Fd\mu_S-\int_{(\R/\Z)^N\times\R^N}Fd\lambdasn\right|\\
\leq2\lip(F)\h(\bfxi)+16\left(4\h(\bfxi)+C\frac{\log(\D(\bfxi)+1)}{\D(\bfxi)}\right)^{\frac12}\sum_{l=1}^N\left\|\wh{\frac{\partial \wF}{\partial\theta_l}}\right\|_{\L^1}\\
\leq\left(2\lip(F)+16\sum_{l=1}^N\left\|\wh{\frac{\partial \wF}{\partial\theta_l}}\right\|_{\L^1}\right)\left(4\h(\bfxi)+C\frac{\log(\D(\bfxi)+1)}{\D(\bfxi)}\right)^{\frac12}.
\end{multline*}
\end{proof}

\begin{rem} The functions in $\F$ are functions with logarithmic
singularities along toric divisors in a toric compactification of
$(\bQ^\times)^N$. The qualitative equidistribution with respect to
this set of test functions is given by the theorem of Chambert-Loir
and Thuillier \cite[Th\'eor\`eme 1.2]{ct}.
\end{rem}


\begin{appendix}

\section{The set of test functions}\label{appendixtestfun}
In this appendix, we show that the test functions in $\F$, when restricted to the unit polycircle $(S^1)^N$, are Haar-integrable as well as their Fourier transforms. We also prove the Haar-integrability of all their first order partial derivatives and their corresponding Fourier transforms.

%

Recall the definition of the test functions. The set $\F$ is given by all real-valued functions $F$ satisfying
\begin{enumerate}[(i)]
\item $F$ is Lipschitz with respect to the distance $\d$ on $(\C^\times)^N$,
\item $\wF(\bftheta)=F(\bftheta,\bfzero)$ is in $\Clf^{N+1}((\R/\Z)^N,\R)$.
\end{enumerate}

The main theorem of this section is:

\begin{teor}\label{teorfun}
For any $F\in\F$, the function $\wF(\bftheta)=F(\bftheta,\bfzero)$ satisfies the following properties:
\begin{enumerate}[(i)]
\item $\wF$ is Haar-integrable,
\item $\wh{\wF}$ is Haar-integrable,
\item for every $l=1,\ldots,N$, $\displaystyle\frac{\partial \wF}{\partial \theta_l}$ are Haar-integrable,
\item for every $l=1,\ldots,N$, $\displaystyle\wh{\frac{\partial \wF}{\partial \theta_l}}$ are Haar-integrable,
\end{enumerate}
\end{teor}

Before proving this result, let us consider a technical lemma. For every function $H:(\R/\Z)^N\longrightarrow\R$ and $\bfalpha=(\alpha_1,\ldots,\alpha_N)\in\{0,1\}^N$, we will use the notation
$$\frac{\partial^{|\bfalpha|}H}{\partial\bftheta^\bfalpha}(\bftheta)=\frac{\partial^{\alpha_1+\ldots+\alpha_N}H}{\partial\theta_1^{\alpha_1}\cdots\theta_N^{\alpha_N}}(\bftheta),$$
whenever it makes sense.

\begin{lemma}\label{lemmatrans}
Let $H:(\R/\Z)^N\longrightarrow\R$ of class $\Clf^{N+1}$ be such that
$$\frac{\partial^{|\bfalpha|}H}{\partial\bftheta^\bfalpha}\in\L^1((\R/\Z)^N)\quad\text{ and }\quad\frac{\partial^{|\bfalpha|+1}H}{\partial\bftheta^\bfalpha\theta_l}\in\L^1((\R/\Z)^N),$$
for $\bfalpha\in\{0,1\}^N$ and $l=1,\ldots,N$. Then, 
$$\wh{\frac{\partial^{|\bfalpha|}H}{\partial\bftheta^\bfalpha}}(\bfn)=\prod_{k=1}^N(2\pi in_k)^{\alpha_k}\wh{H}(\bfn)\quad\text{ and }\quad\frac{\partial^{|\bfalpha|+1}H}{\partial\bftheta^\bfalpha\theta_l}=(2\pi in_l)\prod_{k=1}^N(2\pi in_k)^{\alpha_k}\wh{H}(\bfn).$$
\end{lemma}

\begin{proof}
This lemma is proved applying recursively the following
\begin{align*}
\wh{\frac{\partial \wF}{\partial \theta_1}}(\bfn) &=\int_{(\R/\Z)^N}\frac{\partial \wF}{\partial \theta_1}(\bftheta)e^{-2\pi i\bfn\cdot\bftheta}d\bftheta\\
&=\int_{(\R/\Z)^{N-1}}\left(\int_{\R/\Z}\frac{\partial \wF}{\partial \theta_1}(\bftheta)e^{-2\pi in_1\theta_1}d\theta_1\right)e^{-2\pi i\sum_{j\neq 1}n_j\theta_j}d\theta_2\ldots d\theta_N\\
&=\int_{(\R/\Z)^{N-1}}\left(2\pi in_1\int_{\R/\Z}\wF(\bftheta)e^{-2\pi in_1\theta_1}d\theta_1\right)e^{-2\pi i\sum_{j\neq 1}n_j\theta_j}d\theta_2\ldots d\theta_N\\
&=(2\pi in_1)\int_{(\R/\Z)^N}\wF(\bftheta)e^{-2\pi i\bfn\cdot\bftheta}d\bftheta\\
&=(2\pi in_1)\wh{\wF}(\bfn).
\end{align*}
\end{proof}

\begin{proof}[Proof of Theorem \ref{teorfun}]
By definition, for every $F\in\F$, the function $\wF$ is of class $N+1$. This is, all its partial derivatives up to order $N+1$ are continuous and, since they are defined on a compact space, they are bounded. Hence, for every $\bfalpha\in\{0,1\}^N$ and every $l=1,\ldots,N$
\begin{equation*}
\frac{\partial^{|\bfalpha|}\wF}{\partial\bftheta^\bfalpha},\frac{\partial^{|\bfalpha|+1}\wF}{\partial\bftheta^\bfalpha\theta_l}\in(\L^1\cap\L^2)((\R/\Z)^N).
\end{equation*}
In particular, we obtain parts (i) and (iii).

Let us prove (ii), we have to see that
$$\sum_{\bfn\in\Z^N}\left|\wh{\wF}(\bfn)\right|<\infty.$$
To do so, we will divide the sum over all $\bfn\in\Z^N$ in several subsets. Let $\alpha\in\{0,1\}$ and set
$$W(\alpha)=\begin{cases}\bfzero & \text{if } \alpha=0,\\ \Z^N\setminus\{\bfzero\} &\text{if }  \alpha=1.\end{cases}$$
For $\bfalpha\in\{0,1\}^N$, set also
$$\bfW(\bfalpha)=W(\alpha_1)\times\ldots\times W(\alpha_N).$$
Hence, we have
$$\sum_{\bfn\in\Z^N}\left|\wh{\wF}(\bfn)\right|=\sum_{\bfalpha\in\{0,1\}^N}\sum_{\bfn\in\bfW(\bfalpha)}\left|\wh{\wF}(\bfn)\right|.$$
For $\bfalpha\in\{0,1\}^N$, we have
$$\sum_{\bfn\in\bfW(\bfalpha)}\left|\wh{\wF}(\bfn)\right|=\sum_{\bfn\in\bfW(\bfalpha)}\prod_{k:\alpha_k\neq0}(2\pi n_k)^{-1}\left|\wh{\frac{\partial^{|\bfalpha|}\wF}{\partial\bftheta^\bfalpha}}(\bfn)\right|.$$
We saw that $\frac{\partial^{|\bfalpha|}\wF}{\partial\bftheta^\bfalpha}\in(\L^1\cap\L^2)((\R/\Z)^N)$ and so, by Plancherel's theorem
$$\left\|\wh{\frac{\partial^{|\bfalpha|}\wF}{\partial\bftheta^\bfalpha}}\right\|_{\L^2(\Z^N)}=\left\|\frac{\partial^{|\bfalpha|}\wF}{\partial\bftheta^\bfalpha}\right\|_{\L^2((\R/\Z)^N)}.$$
Using Cauchy-Schwartz inequality, we obtain
\begin{multline*}
\left(\sum_{\bfalpha\in\{0,1\}^N}\sum_{\bfn\in\bfW(\bfalpha)}\left|\wh{\wF}(\bfn)\right|\right)^2=\left(\sum_{\bfalpha\in\{0,1\}^N}\sum_{\bfn\in\bfW(\bfalpha)}\prod_{k:\alpha_k\neq0}(2\pi n_k)^{-1}\left|\wh{\frac{\partial^{|\bfalpha|}\wF}{\partial\bftheta^\bfalpha}}(\bfn)\right|\right)^2\\
\leq\left(\sum_{\bfalpha\in\{0,1\}^N}\sum_{\bfn\in\bfW(\bfalpha)}\prod_{k:\alpha_k\neq0}\frac1{4\pi^2 n_k^2}\right)\left(\sum_{\bfalpha\in\{0,1\}^N}\sum_{\bfn\in\bfW(\bfalpha)}\left|\wh{\frac{\partial^{|\bfalpha|}\wF}{\partial\bftheta^\bfalpha}}(\bfn)\right|^2\right)<\infty.
\end{multline*}

Part (iv) of the theorem is proved by applying the same argument to the function $\frac{\partial\wF}{\partial\theta_l}$ for every $l=1,\ldots,N$.
\end{proof}

\section{Bounds for the Lipschitz constant of the function $f_\delta$}\label{appendixfdelta}
In this appendix, we give a bound for the Lipschitz constant with respect to the spherical distance of the function \mbox{$f_\delta:\P^1(\C)\rightarrow\C$} defined by
$$f_\delta(0:1)=0\ \text{ and }\ f_\delta(1:z)=\rho_\delta(|z|)\frac{z}{|z|} \text{ for any }z\in\C,$$
where $\rho_\delta:\R\rightarrow[0,1]$, with $0<\delta<1$, is given by
\begin{equation*}
\rho_\delta(r)=\begin{cases}0 & \text{if } r<\frac{\delta}{2},\\ \frac{(5\delta-4r)(\delta-2r)^2}{\delta^3} & \text{if } \frac{\delta}{2}\leq r\leq\delta, \\ 1 & \text{if } \delta<r<\frac{1}{\delta}, \\ (-2+\delta r)^2(-1+2\delta r) & \text{if } \frac{1}{\delta}\leq r\leq\frac{2}{\delta}, \\ 0 & \text{if } r>\frac{2}{\delta}.  \end{cases}
\end{equation*}

First we prove that $f_\delta\in\Clf^1(\P^1(\C),\C)$. After\-wards, we will study the Lipschitz constant of its real and imaginary parts. Let us define the usual charts in $\P^1(\C)$,
$$U_0:=\{(z_0:z_1)\in\P^1(\C):z_0\neq0\}\quad\text{ and }\quad U_1:=\{(z_0:z_1)\in\P^1(\C):z_1\neq0\}.$$
It is easy to see that the function $f_\delta$ is compactly supported on $U_0\cap U_1$. In fact, we have that
$$\supp(f_\delta)=\left\{(1:z):\frac{\delta}{2}\leq|z|\leq\frac2\delta\right\}.$$
For this reason, to prove that $f_\delta$ is in $\Clf^1(\P^1(\C),\C)$, it is enough to prove that the function $\rho_\delta(|z|)\frac{z}{|z|}$ is of class $\Clf^1$ in a neighborhood of the set $\left\{z:\frac{\delta}{2}\leq|z|\leq\frac2\delta\right\}$. 

The piecewise-defined function $\rho_\delta$ is continuous, as well as its derivative, which is given by
\begin{equation*}
\rho_\delta'(r)=\begin{cases} -\frac{24}{\delta^3}(\delta-2r)(\delta-r) &\text{if } \frac{\delta}{2}\leq r\leq\delta, \\ 6\delta(-2+\delta r)(-1+\delta r) &\text{if } \frac{1}{\delta}\leq r\leq \frac{2}{\delta}, \\ 0 & \text{otherwise.}
\end{cases}\end{equation*}
Hence, since $|z|$ and $z/|z|$ are smooth on $\C^\times$, we conclude that $\rho_\delta(|z|)\frac{z}{|z|}$ is of class $\Clf^1$.

\begin{lemma}\label{lemmafdelta}
Let $f_\delta$ be defined as above, and set $f_\delta=u_\delta+iv_\delta$. Then, 
$$\lipsph(u_\delta),\lipsph(v_\delta)\leq2\sqrt{2}\frac{\delta^2+9}{\delta^3}.$$\
\end{lemma} 

The spherical distance $\dsph$ on $\P^1(\C)$ can be computed as follows
$$\dsph(p,p'):=2\arccos\left(\frac{|z_0\overline{z_0'}+z_1\overline{z_1'}|}{\sqrt{|z_0|^2+|z_1|^2}\sqrt{|z_0'|^2+|z_1'|^2}}\right),$$
for $p=(z_0:z_1)$ and $p'=(z_0':z_1')$ in $\P^1(\C)$.

To simplify the computations, we will work with an equivalent distance, the {\em chordal distance} $\dch$ on $\P^1(\C)$, which is given by the length of the chord joining two points of $S^2$. For $p=(z_0:z_1)$ and $p'=(z_0':z_1')$ in $\P^1(\C)$, we have
$$\dch(p,p'):=\frac{2|z_0z_1'-z_1z_0'|}{\sqrt{|z_0|^2+|z_1|^2}\sqrt{|z_0'|^2+|z_1'|^2}}.$$

These distances can be compared as follow:
\begin{lemma}\label{lemmadist}
For every $p,p'\in\P^1(\C)$,
$$\frac2\pi\dsph(p,p')\leq\dch(p,p')\leq\dsph(p,p').$$
\end{lemma}

\begin{proof}
We work on the sphere using the stereographic projection. Since the chordal distance $\dch$ between two points in the sphere is the length of the chord joining them and the spherical distance $\dsph$ is the angle between the vectors both points define, we have
$$\dch(p,p')=2\sin\left(\frac{\dsph(p,p')}{2}\right),\text{ for every }p,p'\in\P^1(\C).$$
For any pair of points, we have $\dsph(p,p')\leq\pi$ so we deduce
$$\dch(p,p')\leq\dsph(p,p').$$

Now, let $\beta>0$ be such that $\beta\dsph(p,p')\leq\dch(p,p')$ for all $p,p'\in\P^1(\C)$. This is equivalent to $\beta x\leq2\sin\left(\frac{x}{2}\right)$ for every $0\leq x\leq\pi$. By the convexity of the function $2\sin\left(\frac{x}{2}\right)$, we deduce that the optimal value is $\beta=\frac{2}{\pi}$.
\end{proof}

\begin{proof}[Proof of Lemma \ref{lemmafdelta}]
Let us compute now a bound for the Lipschitz constants, with respect to the spherical distance, of the $u_\delta$ and $v_\delta$. To do so, we choose coordinates $(x,y)$ in $\R^2\cong\C$. Let
$$\tilde{u}_\delta(x,y):=u_\delta(1:x+iy)=\frac{\rho_\delta(\sqrt{x^2+y^2})}{\sqrt{x^2+y^2}}x,$$
$$\tilde{v}_\delta(x,y):=v_\delta(1:x+iy)=\frac{\rho_\delta(\sqrt{x^2+y^2})}{\sqrt{x^2+y^2}}y.$$
Since the computations are symmetric for both the real and imaginary parts of $f_\delta$, it is enough to study the Lipschitz constant of one of them. To simplify these computations, we will study the Lipschitz constant with respect to the chordal distance in the Riemann sphere and conclude by applying the comparison between the chordal and spherical distances.

First of all, recall that the chordal distance restricted to the open subset $U_0\subset\P^1(\C)$ is given by
$$\dch((1:x_0+iy_0),(1:x_1+iy_1))=\frac{2\|(x_0,y_0)-(x_1,y_1)\|}{\sqrt{1+m(x_0,y_0)^2}\sqrt{1+m(x_1,y_1)^2}},$$
where $\|\cdot\|$ denotes the Euclidean metric on $\R^2$ and $m(x,y)=\sqrt{x^2+y^2}$.
Now, since the function $u_\delta$ is supported on $U_0$, we have
\begin{multline*}
\sup_{z_0,z_1\in\C}\frac{|u_\delta(1:z_0)-u_\delta(1:z_1)|}{\dch((1:z_0),(1:z_1))}\\
=\sup_{(x_0,y_0),(x_1,y_1)\in\R^2}\frac{|\tilde{u}_\delta(x_0,y_0)-\tilde{u}_\delta(x_1,y_1)|}{\|(x_0,y_0)-(x_1,y_1)\|}\frac{\sqrt{1+m(x_0,y_0)^2}\sqrt{1+m(x_1,y_1)^2}}{2}.
\end{multline*}

We consider different cases.

\begin{enumerate}[1.]
\item If $(x_0,y_0),(x_1,y_1)\notin D(0,\frac2\delta)$, we trivially obtain
$$\frac{|\tilde{u}_\delta(x_0,y_0)-\tilde{u}_\delta(x_1,y_1)|}{\dch((1:x_0+iy_0),(1:x_1+iy_1))}=0.$$
\item Suppose $(x_0,y_0),(x_1,y_1)\in D(0,\frac2\delta)$. For $t\in[0,1]$, consider the function $g(t)=\tilde{u}_\delta((1-t)(x_0,y_0)+t(x_1,y_1))$. By the mean value theorem, there is some $c\in(0,1)$ such that $g(1)-g(0)=g'(c)$. Applying the chain rule, we obtain
$$\tilde{u}_\delta(x_1,y_1)-\tilde{u}_\delta(x_0,y_0)=\nabla\tilde{u}_\delta((1-c)(x_0,y_0)+c(x_1,y_1))\cdot(x_1-x_0,y_1-y_0).$$
Hence, we deduce
\begin{equation}\label{mvthmr2}
\frac{|\tilde{u}_\delta(x_0,y_0)-\tilde{u}_\delta(x_1,y_1)|}{\|(x_0,y_0)-(x_1,y_1)\|}\leq\sup_{(x,y)\in D(0,\frac2\delta)}\|\nabla\tilde{u}_\delta(x,y)\|.
\end{equation}
Let us study the gradient of $\tilde{u}_\delta$. For every $(x,y)\in\R^2$ we have
\begin{align*}
\frac{\partial\tilde{u}_\delta}{\partial x}(x,y)& =\left(\frac{x}{m(x,y)}\right)^2\rho_\delta'(m(x,y))+\left(\frac{y}{m(x,y)}\right)^2\frac{\rho_\delta(m(x,y))}{m(x,y)},\\
\frac{\partial\tilde{u}_\delta}{\partial y}(x,y)&=\frac{xy}{m(x,y)^2}\left(\rho_\delta'(m(x,y))-\frac{\rho_\delta(m(x,y))}{m(x,y)}\right).
\end{align*}
Without loss of generality, we restrict ourselves to the situation where $(x,y)$ verifies $\frac\delta2\leq m(x,y)\leq\frac2\delta$, since otherwise both partial derivatives would vanish. It can be easily shown that $|\rho'_\delta(r)|\leq\frac3\delta$ for every $r\geq0$. This, together with the fact that $0\leq\rho_\delta\leq1$, $x\leq m(x,y)$, $y\leq m(x,y)$ and $m(x,y)\geq\frac\delta2$, leads to
$$\left|\frac{\partial\tilde{u}_\delta}{\partial x}(x,y)\right|,\left|\frac{\partial\tilde{u}_\delta}{\partial y}(x,y)\right|\leq\frac4\delta.$$
We then conclude that, for any $(x,y)\in\R^2$,
$$\|\nabla\tilde{u}_\delta(x,y)\|\leq\frac{4\sqrt{2}}{\delta}.$$

On the other hand, given $(x_0,y_0),(x_1,y_1)\in D(0,\frac2\delta)$ we have that
$$\frac{\sqrt{1+m(x_0,y_0)^2}\sqrt{1+m(x_1,y_1)^2}}{2}\leq\frac{\delta^2+4}{2\delta^2}.$$
Therefore, we obtain
\begin{equation*}
\frac{|\tilde{u}_\delta(x_0,y_0)-\tilde{u}_\delta(x_1,y_1)|}{\dch((1:x_0+iy_0),(1:x_1+iy_1))}\leq
2\sqrt{2}\frac{\delta^2+4}{\delta^3}.
\end{equation*}

\item Suppose now that $(x_0,y_0)\in D(0,\frac2\delta)$ and $(x_1,y_1)\in D(0,\frac3\delta)\setminus D(0,\frac2\delta)$. As we did in the previous case, we can deduce that
$$\frac{|\tilde{u}_\delta(x_0,y_0)-\tilde{u}_\delta(x_1,y_1)|}{\|(x_0,y_0)-(x_1,y_1)\|}\leq\frac{4\sqrt{2}}{\delta}$$
and
$$\frac{\sqrt{1+m(x_0,y_0)^2}\sqrt{1+m(x_1,y_1)^2}}{2}\leq\frac{\delta^2+9}{2\delta^2}.$$
Hence, we obtain
\begin{equation*}
\frac{|\tilde{u}_\delta(x_0,y_0)-\tilde{u}_\delta(x_1,y_1)|}{\dch((1:x_0+iy_0),(1:x_1+iy_1))}\leq
2\sqrt{2}\frac{\delta^2+9}{\delta^3}.
\end{equation*}

\item Finally suppose that $(x_0,y_0)\in D(0,\frac2\delta)$ and $(x_1,y_1)\notin D(0,\frac3\delta)$. In this situation, we have $\tilde{u}_\delta(x_1,y_1)=0$ and
$$|\tilde{u}_\delta(x_1,y_1)|=|\rho_\delta(m(x_0,y_0))|\frac{|x_0|}{m(x_0,y_0)}\leq1.$$
Since 
$$\dch((1:x_0+iy_0),(1:x_1+iy_1))\geq\dch((1:\frac2\delta),(1:\frac3\delta))=\frac{2\delta}{\sqrt{(\delta^2+9)(\delta^2+4)}},$$
we conclude
\begin{equation*}
\frac{|\tilde{u}_\delta(x_0,y_0)-\tilde{u}_\delta(x_1,y_1)|}{\dch((1:x_0+iy_0),(1:x_1+iy_1))}\leq
\frac{2\delta}{\sqrt{(\delta^2+9)(\delta^2+4)}}{2\delta}.
\end{equation*}
\end{enumerate}

Having studied all these cases, we deduce that
$$\sup_{(x_0,y_0),(x_1,y_1)\in\R^2}\frac{|\tilde{u}_\delta(x_0,y_0)-\tilde{u}_\delta(x_1,y_1)|}{\dch((1:x_0+iy_0),(1:x_1+iy_1))}\leq2\sqrt{2}\frac{\delta^2+9}{\delta^3}.$$

As we mentioned above, we were looking for a bound of the Lipschitz constant of $u_\delta$ with respect to the spherical distance. By Lemma \ref{lemmadist}, we know that $\dsph(p,p')\geq\dch(p,p')$ for any pair of points $p,p'\in\P^1(\C)$ and we obtain
\begin{multline*}
\lipsph(u_\delta)=\sup_{p,p'\in\P^1(\C)}\frac{|u_\delta(p)-u_\delta(p')|}{\dsph(p,p')}\\
\leq\sup_{(x_0,y_0),(x_1,y_1)\in\R^2}\frac{|\tilde{u}_\delta(x_0,y_0)-\tilde{u}_\delta(x_1,y_1)|}{\dch((1:x_0+iy_0),(1:x_1+iy_1))}\leq2\sqrt{2}\frac{\delta^2+9}{\delta^3}.
\end{multline*}

Analogously, we deduce that $\lipsph(v_\delta)\leq2\sqrt{2}\frac{\delta^2+9}{\delta^3}$.
\end{proof}

\end{appendix}

\providecommand{\bysame}{\leavevmode\hbox to3em{\hrulefill}\thinspace}
\providecommand{\MR}{\relax\ifhmode\unskip\space\fi MR }
\providecommand{\MRhref}[2]{%
  \href{http://www.ams.org/mathscinet-getitem?mr=#1}{#2}
}
\providecommand{\href}[2]{#2}

\end{document}